\documentclass[11pt]{amsart}

\usepackage{graphicx} 
\usepackage{amsfonts,amsmath,amssymb,amsthm}
\usepackage{mathrsfs}
\usepackage{bm}
\usepackage{todonotes} 
\usepackage{hyperref,cleveref}
\usepackage{xcolor}
\usepackage{lineno}
\usepackage{soul}

\theoremstyle{plain}

\newtheorem{Theorem}{Theorem}[section]
\newtheorem{Lemma}[Theorem]{Lemma}
\newtheorem{Proposition}[Theorem]{Proposition}
\newtheorem{Corollary}[Theorem]{Corollary}

\theoremstyle{definition}
\newtheorem{Definition}{Definition}[section]
\newtheorem{Example}{Example}[section]
\newtheorem{Question}{Question}

\newtheorem{Remark}{Remark}

\newcommand{\Tq}{\mathcal{T}_{\text{qua}}}

\newcommand{\nor}[1]{[\![ #1 ]\!]}

\DeclareMathOperator{\Span}{span}
\DeclareMathOperator{\Star}{star}

\newcommand*{\abs}[1]{\left\lvert#1\right\rvert}

\definecolor{mypink}{RGB}{255, 204, 218}

\newcommand{\revise}[1]{{#1}}

\numberwithin{equation}{section}
\title{The \revise{sharpness} 
condition for constructing a finite element from a superspline}

\author{Jun Hu}
\address{LMAM and School of Mathematical Sciences, Peking University, Beijing 100871, P.R. China.}
\email{hujun@math.pku.edu.cn}

\author{Ting Lin}
\address{School of Mathematical Sciences, Peking University, Beijing 100871, P.R. China.}
\email{lintingsms@pku.edu.cn}

\author{Qingyu Wu}
\address{School of Mathematical Sciences, Peking University, Beijing 100871, P.R. China.}
\email{wu\_qingyu@pku.edu.cn}

\author{Beihui Yuan}
\address{Beijing Institute of Mathematical Sciences and Applications, Beijing 101408, P.R. China.}
\email{beihuiyuan@bimsa.cn}

\thanks{Jun Hu was supported by NSFC project No.12288101. Ting Lin was supported by NSFC project No.123B2014. \revise{Beihui Yuan was supported by Beijing Natural Science Foundation No. 1254042.}}

\subjclass[2020]{65N30, 65D07}

\keywords{Finite element method, spline, continuity vector, pre-element mapping, extendability}

\begin{document}

\begin{abstract}
    This paper addresses \revise{sharpness} conditions for constructing $C^r$ conforming finite element spaces from a superspline spaces on general simplicial triangulations. We introduce the concept of extendability for the pre-element spaces, which encompasses both the superspline spaces and the finite element spaces. By examining the extendability condition for both types of spaces, we provide an answer to the conditions regarding the construction. A corollary of our results is that constructing $C^r$ conforming elements in $d$ dimensions generally requires an extra $C^{2^{s}r}$ continuity on $s$-codimensional simplices, and the polynomial degree is at least $(2^d r + 1)$.
\end{abstract}


\maketitle

\section{Introduction} 
\label{sec:intro}

The finite element method and the spline method are two closely related methods using piecewise polynomials to approximate the target function.
The main distinction between these two methods is their own construction philosophies: The \emph{finite element} follows a \emph{bottom-up} approach, while the \emph{spline} employs a \emph{top-down} approach. 
In fact, in the finite element method, we first propose the local shape function space and the corresponding degrees of freedom.
Then, we glue all the local spaces together by matching the degrees of freedom from adjacent simplices to obtain the global finite element space.
In contrast, in the spline method, we start with a global piecewise polynomial space, namely the spline space, on which some inter-element continuity conditions are imposed. 
Both methods have been widely used. In particular, the finite element method has been popularly used in solving partial differential equations, while the spline method has been largely used in data fitting and 3D modeling. 

For piecewise polynomial spaces with $C^0$ continuity, the finite element space is essentially equivalent to the spline space, which is spanned by the Lagrange basis. However, for general cases, these two methods result in different spaces. Both approaches will be briefly introduced below.

The $C^r$ spline space is the intersection of the piecewise polynomial space and the $C^{r}$ continuous function space. This top-down definition results in a lack of direct control on the local smoothness of $C^r$ spline functions. 
For example, the splines might have higher smoothness locally, known as (intrinsic) supersmoothness in the literature, see \cite{sorokina2010intrinsic, floater2020characterization}. 
When $r \ge 1$, the dimension of the $C^r$ spline space relies heavily on the global geometry of the triangulation. 
A typical case is the Morgan-Scott triangulation~\cite[Figure 9.2]{lai2007spline}, which exemplifies the dependence. 
Therefore, sometimes the spline space does not have a local basis~\cite{alfeld2000nonexistence}.
As a result, to determine the dimension of the spline spaces with $r\geq 1$, even for triangulations of planar regions, is a highly non-trivial task~\cite{schenck2016algebraic}. 
Nonetheless, bivariate spline spaces have been intensively studied, for which there is a well-known lower bound for the dimension of spline spaces derived by Schumaker in \cite{schumaker1979dimension}. 
It has also been proved, for example in \cite{alfeld1987dimension,alfeld1990dimension,dong1991spaces}, that under some mild  conditions and assumptions on degrees, Schumaker's formula gives the actual dimension of spline spaces, see also 
\revise{
\cite{morgan1975nodal,ibrahim1991super} for local bases. For trivariate splines, less is known, see \cite{alfeld1993generic,dipasquale2021lower} and references therein, see also \cite{alfeld1992dimension} for local bases.
}

In contrast, the study of the $C^r$ finite element spaces concentrates on how to use local degrees of freedom to achieve the continuity condition. 
In two dimensions, \cite{bramble1970triangular} generalizes the $C^1$ Argyris element \cite{argyris1968tuba} to the general $C^r$ elements, where the partial derivatives up to order $2r$ are employed in the degrees of freedom.
In three dimensions, the existing constructions of $C^r$ finite element spaces require: (1) the partial derivatives up to order $4r$ at each vertex, and (2) the partial derivatives up to order $2r$ on each edge. See \cite{ vzenivsek1970interpolation,zhang2009family} for the elements of $r = 1$, \cite{zhang2016family} for the elements of $r = 2$, and \cite{lai2007spline,shi1988Higher,wang2013multivariate} for the elements of arbitrary $r$. In arbitrary $d$ dimensions, it is commonly conjectured that the $C^r$ finite elements can be constructed if the $\mathcal{P}_{2^{d}r+1}$ polynomial space is chosen as the local shape function space, and the derivatives up to $2^{d-1-s}r$ order on $s$-dimensional subsimplex are used. Recently, three of the authors \cite{hu2021construction} gave the first construction of the conforming $C^r$ finite element in $d$ dimensions with respect to the simplicial triangulation, based on the above conditions. 

This paper sheds light on the distinction and connection between finite elements and spline spaces. 
Our results indicate that the continuity requirement of the finite element construction is tight. 
\revise{
To this end, we recall the definition of the finite element spaces and superspline spaces, which can be found in standard textbooks \cite{brenner2008mathematical,ern2021finite,MR1900298,MR2348176}.
}

\subsection*{Finite Element spaces and Superspline spaces}

Throughout this paper, following \cite{ern2021finite}, we assume that a \emph{finite element} consists of a \emph{Ciarlet's triple} $(K, P, \Sigma)$ where
\begin{itemize}
    \item[-] $K$ is a $d$-simplex, embedded in $\mathbb{R}^{d}$;
    \item[-] $P$ is a finite-dimensional space of polynomial functions;
    \item[-] $\Sigma$ is a basis for the dual space $P^{\vee}$ (the space of linear \revise{functionals} from $P$ to $\mathbb R$), together with a natural partition $\Sigma = \coprod_{F \subseteq K} \Sigma_{F}$ with respect to the subsimplices $F$ of $K$. The element in $\Sigma$ is called the \emph{degree of freedom}.
\end{itemize}

Given a triangulation $\mathcal{T}$ of the underlying region $\Omega$, we can define a \emph{finite element space} $\bm{E}(\mathcal{T})$ as a subspace of $L^{2}(\Omega)$, by specifying a finite element $(K_{j}, P, \Sigma(K_{j}))$ for each $d$-simplex $K_{j} \in \mathcal{T}$, satisfying the condition that for each subsimplex $F \subseteq K_{j}\cap K_{j'}$, $\Sigma_{F}(K_{j}) = \Sigma_{F}(K_{j'})$, which is shortened as $\Sigma_{F}$. Hence, we have a \emph{global degree of freedom} $\Sigma(\mathcal{T}) := \coprod_{F \in \mathcal{T}} \Sigma_{F}$, according to the above condition. Then the global finite element space is defined as
\begin{equation}
    \label{eq:def-FEMSp}
    \begin{aligned}
        \bm{E}(\mathcal{T}) & := \{u : \text{for each } d \text{-simplex } K \in \mathcal{T}, u|_{K}\in P(K); \\
        & \qquad \qquad \text{for each proper subsimplex } F \in \mathcal{T} \\
        & \qquad \qquad \text{and each degree of freedom } l \in \Sigma_{F}, \\
        & \qquad \qquad l(u|_{K}) \text{ is single-valued for any } K \in \Star(F; \mathcal{T})\}.
    \end{aligned}
\end{equation}
Here the $d$-simplex set $\Star(F; \mathcal{T})$ is defined as
\begin{equation*}
    \Star(F; \mathcal{T}) := \{K \in \mathcal{T}: K \text{ is a } d \text{-simplex}, F \text{ is a subsimplex of } K\}.
\end{equation*}

\revise{
Now we introduce the concepts of splines and supersplines. For many applications in the spline theory, it is often beneficial to consider splines with enhanced smoothness conditions on subsimplices. In \cite[Chapter 5.5]{lai2007spline}, such smoothness conditions primarily focus on vertices and edges, as the focus is restricted to bivariate and trivariate splines. The superspline space in any dimension is defined as follows. Let $\nabla^n f$ denote the $n$-th order gradients of a scalar function $f: \mathbb R^d \to \mathbb R$, by successively taking the gradient operation $n$ times. Note that $\nabla^n f$ is a symmetric $n$-th order tensor-valued function, and its components are given by appropriate $n$-th order partial derivatives.
}

\begin{Definition}
For a given \emph{continuity vector} $\bm{r} = (r_{1}, \cdots, r_{d})$ and a given \emph{polynomial degree} $k$, define the superspline space $\bm{S}^{\bm{r}}_{k}(\mathcal{T})$ over the triangulation $\mathcal{T}$ as
\begin{equation}
    \label{eq:def_srk}
    \begin{aligned}
        \bm{S}^{\bm{r}}_{k}(\mathcal{T}) & := \{u  : \text{for each } d \text{-simplex } K \in \mathcal{T}, u|_{K} \in \mathcal{P}_{k}(K); \\
        & \qquad \qquad \text{for each proper subsimplex } F \in \mathcal{T}_{d - s} \text{ and each } 0 \le n \le r_{s}, \\
        & \qquad \qquad \nabla^n u|_{F} \text{ is single-valued for any } K \in \Star(F; \mathcal{T})\}.
    \end{aligned}
\end{equation}
Define
\begin{equation}
    \label{eq:def_sr}
    \begin{aligned}
        \bm{S}^{\bm{r}}(\mathcal{T}) & := \{u  : \text{for each } d \text{-simplex } K \in \mathcal{T}, u|_{K} \in \mathcal{P}(K); \\
        & \qquad \qquad \text{for each proper subsimplex } F \in \mathcal{T}_{d - s} \text{ and each } 0 \le n \le r_{s}, \\
        & \qquad \qquad \nabla^n u|_{F} \text{ is single-valued for any } K \in \Star(F; \mathcal{T})\}.
    \end{aligned}
\end{equation}
be the union of all $\bm{S}_{k}^{\bm r}(\mathcal{T})$ for $k \ge 0$. 
Here, $\mathcal{P}_{k}(K)$ denotes the polynomial function space defined on $K$ with degrees less than or equal to $k$, $\mathcal{P}(K)$ denotes the polynomial function space defined on $K$, 
\revise{
and $\mathcal{T}_{d - s}$ is the collection of $(d - s)$-dimensional subsimplices of $\mathcal{T}$.
} 
Hereafter, we always assume that $r_{1} \le r_{2} \le \cdots \le r_{d}$ for convenience.
\end{Definition}

Similar to spline spaces, the dimension counting and the basis construction of superspline spaces have been explored, see \cite{alfeld1991structure,lai2007spline,toshniwal2023algebraic}.
\revise{
\begin{Remark}
We would like to emphasize that the extra smoothness condition is explicitly stated in the definition of the superspline spaces. Since a lower dimensional subsimplex is included in a higher dimensional subsimplex, it is sufficient to consider the superspline spaces with $r_{1} \le r_{2} \le \cdots \le r_{d}$. It is possible that two superspline spaces nontrivially coincide with different continuity vectors and polynomial degrees. This is known as intrinsic supersmoothness, as discussed in \cite{sorokina2010intrinsic, floater2020characterization}. However, a detailed discussion of this supersmoothness is beyond the scope of this paper.
\end{Remark}
}

There are strong relationships between finite element spaces and superspline spaces (see, e.g., \cite{schumaker1989super}). The construction of the Argyris elements then indicates that the superspline space $\bm{S}^{(1, 2)}_{k}(\mathcal{T})$ \revise{for $k \ge 5$} can be regarded as a global finite element space, or equivalently speaking, the superspline space $\bm{S}^{(1, 2)}_{k}(\mathcal{T})$ admits a finite element construction. Many classical constructions of $C^r$ finite elements falls into this category. The global space of \cite{bramble1970triangular} is $\bm{S}^{(r, 2r)}_{4r + 1}(\mathcal{T})$ in two dimensions; the global space of \cite{zhang2009family} is $\bm{S}^{(1,2,4)}_{k}(\mathcal{T})$ for $k\ge 9$ in three dimensions. 
\revise{
The above construction also exists in the language of the spline theory, see \cite{lai2007spline} for more details. Recently, in \cite{hu2021construction}, three of the authors constructed the $C^{r_1}$ conforming finite elements, by using shape function space $\mathcal{P}_{k}$ and ${r_s}$-th order derivatives in $(d-s)$-dimensional subsimplices for
\begin{equation}
    \label{eq:assumption}
    k \ge 2r_{d} + 1, \quad r_{d} \ge 2r_{d - 1} \ge \cdots \ge 2^{d - 1} r_{1}.
    \tag{A1}
\end{equation}
We shall prove in the appendix that the global finite element spaces of these finite elements \cite{hu2021construction} is $\bm{S}^{\bm{r}}_{k}(\mathcal{T}) = \bm{S}^{(r_1, \cdots, r_d)}_{k}(\mathcal{T})$. 
}
This is precisely why we incorporate superspline spaces into our study. 

The problem of constructing $C^r$ finite element spaces from superspline spaces is then formulated as the following question. 
\begin{Question}
    \label{problem:main}
    When does a superspline space $\bm S_k^{\bm r}\revise{(\mathcal{T})}$ admit a finite element construction?
\end{Question}

\revise{
By \cite{hu2021construction} and the discussion before, \eqref{eq:assumption} is a sufficient condition. The goal of this paper is to show that \eqref{eq:assumption} is also necessary. The study of smoothness conditions in arbitrary dimensions can be traced back at least to \cite{alfeld1992dimension}, which provided a sufficient condition for a space to possess a local basis. Subsequently, many of the necessary conditions regarding the independence of the determinant sets have been proposed, especially in some macro elements in two dimensions. Relevant details can be found in \cite[Chapter 7.8]{lai2007spline} and the references therein. The conditions presented in those works are closely associated with the Bernstein-Bézier representations. Nevertheless, in this paper, we approach this problem from a different perspective. Instead of relying on expansions, we aim to find some intrinsic properties of the finite element spaces and the superspline spaces. 
} 
It should be also noted that having a local basis is not a sufficient condition for admitting a finite element construction in the sense of Ciarlet's triple. Morgan and Scott \cite{morgan1975nodal} provided a nodal basis of $\bm{S}^{(1, 1)}_{k}(\mathcal{T})$ in two dimensions for $k \ge 5$. We shall discuss this case in \Cref{ex:2D}.

\subsection*{Contributions and main techniques}

This paper solves \Cref{problem:main}, by characterizing the sufficient and in some sense necessary condition for a superspline space to be a global finite element space. The key idea is to introduce the concept of extendability. To this end, we first introduce the pre-element space mappings, 
\revise{
and then give a precise definition of \Cref{problem:main}.
} 
\begin{Definition}
    A \emph{pre-element space mapping} \revise{$\mathsf P$} in $d$ dimensions is a vector space-valued mapping \revise{$\mathcal{T} \mapsto \bm P(\mathcal{T})$} defined for all the triangulations $\mathcal{T}$ of $\Omega \subset \mathbb{R}^{d}$ satisfying the following conditions:
    \begin{enumerate}
        \item Given a triangulation $\mathcal{T} := \mathcal{T}(\Omega)$, $\bm{P}(\mathcal{T})$ is 
        \revise{
        a space of functions defined in $\Omega$.
        }
        \item If $\mathcal{T}'$ is a subtriangulation (namely, a pure $d$ subcomplex) of $\mathcal{T}$, then the restriction of $\bm{P}(\mathcal{T})$ on $\mathcal{T}'$ is a subspace of $\bm{P}(\mathcal{T}')$.
    \end{enumerate}
\end{Definition}

By the definition, the finite element space \revise{mappings $\mathsf E: \mathcal{T} \mapsto \bm E(\mathcal{T})$}, the superspline space \revise{mappings $\mathsf S_{k}^{\bm r} : \mathcal{T} \mapsto \bm S_{k}^{\bm r}(\mathcal{T})$ and $\mathsf S^{\bm r} : \mathcal{T} \mapsto \bm S^{\bm r}(\mathcal{T})$} are special cases of the pre-element space mappings. 
\revise{
Note that the space $\bm P(\mathcal{T})$ can be infinite dimensional, for example, for each $\mathcal{T} = \mathcal{T}(\Omega)$, the continuous function space $C(\Omega)$, the $r$-th differentiable function space $C^r(\Omega)$, and the smooth function $C^{\infty}(\Omega)$. Both the concept itself and the given examples here are strongly related to the sheaf theory, see \cite{bredon2012sheaf}.

The next definition gives a precise definition about the term ``admitting a finite element construction" in \Cref{problem:main}.
}

\begin{Definition}
We say a pre-element space mapping \revise{$\mathsf P : \mathcal{T} \mapsto \bm P(\mathcal{T})$} admits a construction of a finite element, if there exists a Ciarlet's triple such that 
\revise{
the corresponding finite element space $\bm E(\mathcal{T})$ satisfies $\bm P(\mathcal{T}) = \bm E(\mathcal{T})$ for all triangulation $\mathcal{T}$, where $\bm E(\mathcal{T})$ is defined by \eqref{eq:def-FEMSp}.
}
\end{Definition}

We now define the extendability for the pre-element space mapping \revise{$\mathsf P$}.

\begin{Definition}
We say a pre-element space mapping \revise{$\mathsf P : \mathcal{T} \mapsto \bm{P}(\mathcal{T})$} is \emph{extendable}, if for any subtriangulation $\mathcal{T}'$ of $\mathcal{T}$, the restriction operator $\bm{P}(\mathcal{T}) \to \bm{P}(\mathcal{T}')$ is onto.
\end{Definition}

\begin{Remark}
    Note that the extendability is proposed for a mapping \revise{$\mathsf P$}, rather than a specific function space $\bm P(\mathcal{T})$.
\end{Remark}

From the bottom-up construction of the global finite element space, we can prove that all the global finite element mappings $\mathsf{E}$ are extendable, see \Cref{thm:FE-extendablity}.

As a corollary, if a superspline space admits a finite element construction, then it must be extendable. We then propose the following question, 
\revise{
which serves as a step towards answering  \Cref{problem:main}.
}

\begin{Question}
    \label{problem:extendable}
    When is a superspline space extendable? 
\end{Question}

\revise{
It follows from \cite{hu2021construction}, \Cref{coro:E=S}, and \Cref{thm:FE-extendablity} that \eqref{eq:assumption} is a sufficient condition.
} 
We will show that the answers to both \Cref{problem:main} and \Cref{problem:extendable} happen to be if and only if Assumption \eqref{eq:assumption} holds, which are formally stated as \Cref{thm:admit} and \Cref{thm:extendable-srk}, respectively.

\begin{Theorem}
    \label{thm:admit}
    Let \revise{$\mathsf S_k^{\bm r} : \mathcal{T} \mapsto \bm S_k^{\bm r}(\mathcal{T})$} be a superspline space mapping. Then, \revise{$\mathsf S_k^{\bm r}$} admits a construction of a finite element if and only if the continuity vector $\bm{r}$ and the polynomial degree $k$ satisfy Assumption \eqref{eq:assumption}.
\end{Theorem}

\revise{
Note that the ``if" part is essentially given by \cite{hu2021construction}, with additional technical details shown in the appendix, see \Cref{coro:E=S}. The ``only if" part is the main contribution of this paper, which is divided to the following two steps:
}
\begin{enumerate}
    \item If the mapping \revise{$\mathsf{S}^{\bm{r}}_{k}$} admits a construction of a finite element, and it must be extendable, see \revise{\Cref{thm:FE-extendablity}.}
    \item The  mapping \revise{$\mathsf{S}^{\bm{r}}_{k}$} is extendable only if Assumption \eqref{eq:assumption} holds, see \Cref{thm:extendable-srk}.
\end{enumerate}

\revise{
Moreover, we show that the concepts of superspline mappings and their extendability can be similarly established in the context of macro elements. Though only some primary results are presented in this paper, the basic framework and some interesting observations are revealed and discussed, which motivate further studies in the future, see \Cref{subsec:Macro}.
}

\subsection*{Notations}
Throughout this paper, we keep the following notations. 

Let $\mathcal{T} = \mathcal{T}(\Omega)$ be a triangulation of a domain $\Omega \subseteq \mathbb{R}^{d}$. The triangulation $\mathcal{T}'$ is a pure $d$ subcomplex of $\mathcal{T}$, called a subtriangulation of $\mathcal{T}$, denoted as $\mathcal{T}' \subseteq \mathcal{T}$. Let $\mathcal{T}_{s}$ be the collection of $s$-dimensional subsimplices of $\mathcal{T}$.

We shall use $K$ to represent a $d$-simplex and $F, E$ to represent subsimplices of $K$. The vertices of $K$ are denoted as $V_{0}, V_{1}, \cdots, V_{d}$. We assume that $\lambda_{0}, \lambda_{1}, \cdots, \lambda_{d}$ are barycentric coordinates corresponding to $V_{0}, V_{1}, \cdots, V_{d}$. Here $\lambda_{i}$ is a linear function, which can be naturally extended to $\mathbb{R}^{d}$, and $\lambda_{i}(V_{i'}) = \delta_{ii'}$ for $i, i' = 0, 1, \cdots, d$, where $\delta$ is Kronecker's delta.

For a $(d - s)$-dimensional subsimplex $F$ of $K$, formed by vertices $V_{s}, \cdots, V_{d}$, we assume that $\lambda_{F, s}, \cdots, \lambda_{F, d}$ are barycentric coordinates corresponding to $V_{s}, \cdots, V_{d}$. Here $\lambda_{F, i}$ is a linear function defined on $F$, such that $\lambda_{F, i} = \lambda_i|_{F}$ for $i = s, \cdots, d$. In particular, when $F$ is a vertex $V_{d}$, we assume that $\lambda_{F, d} = 1$.
\section{Extendability: An example}
\label{sec:extendable}

This section focuses on the extendability. We first show that the finite element space mappings are always extendable, by leveraging the bottom-up construction.

\begin{Proposition}
    \label{thm:FE-extendablity}
    The finite element space mapping \revise{$\mathsf E: \mathcal{T} \mapsto \bm E(\mathcal{T})$ is extendable, where $\bm E(\mathcal{T})$ is defined in \eqref{eq:def-FEMSp}.}
\end{Proposition}

\begin{proof}
    The proof is based on a direct construction. For any subtriangulation $\mathcal{T}' \subseteq \mathcal{T}$, it suffices to prove that the restriction operator $\bm{E}(\mathcal{T}) \to \bm{E}(\mathcal{T}')$ is onto. For any $u \in \bm{E}(\mathcal{T}')$, consider $v \in \bm{E}(\mathcal{T})$ such that
    \begin{itemize}
        \item[-] $l(v) = l(u)$ for each degree of freedom $l \in \Sigma_{F}$ defined on each subsimplex $F \in \mathcal{T}' \subseteq \mathcal{T}$. Note that the right-hand side $l(u)$ is well-defined since $l(u|_{K})$ is single-valued for any $K \in \Star(F; \mathcal{T}')$. 
        \item[-] $l(v) = 0$ for each degree of freedom $l \in \Sigma_{F}$ defined on each subsimplex $F \in \mathcal{T}$ but $F \notin \mathcal{T}'$.
    \end{itemize}
    By the above choice, it follows that $u|_{K} = v|_{K}$ holds for all $d$-simplex $K \in \mathcal{T}'$, by the definition of the finite element spaces, namely, $l(v|_{K}) = l(u|_{K})$ holds for each $l \in \Sigma_{F}$, $F \subseteq K \in \mathcal{T}'$.
\end{proof}

\revise{
By \Cref{coro:E=S}, the superspline space mapping $\mathsf{S}^{\bm{r}}_{k}$ is equal to a finite element space mapping $\mathsf{E}^{\bm{r}}_{k}$ under Assumption \eqref{eq:assumption}, which admits a construction of a finite element. Therefore, by \Cref{thm:FE-extendablity} we have the following corollaries.
}

\begin{Corollary}
    Under Assumption \eqref{eq:assumption}, the superspline space mapping \revise{$\mathsf{S}^{\bm{r}}_{k}$} is extendable.
\end{Corollary}

\begin{Corollary}
    Under Assumption
    \begin{equation}
        \label{eq:assumption2}
        r_{d} \ge 2r_{d - 1} \ge \cdots \ge 2^{d - 1} r_{1},
        \tag{A2}
    \end{equation}
    the superspline space mapping \revise{$\mathsf{S}^{\bm{r}}$} is extendable.
\end{Corollary}

\begin{proof}
    For each function in $u \in \bm{S}^{\bm{r}}(\mathcal{T}')$, there exists a $k \ge 2r_d + 1$ such that $u \in \bm{S}^{\bm{r}}_{k}(\mathcal{T}')$. We can extend $u$ as a function in $\bm{S}^{\bm{r}}_{k}(\mathcal{T}) \subseteq \bm{S}^{\bm{r}}(\mathcal{T})$ by the extendablity of \revise{$\mathsf{S}^{\bm{r}}_{k}$}, since the continuity vector $\bm{r}$ and the polynomial degree $k$ satisfy Assumption \eqref{eq:assumption}. Therefore, \revise{the superspline space mapping $\mathsf{S}^{\bm{r}}$} is extendable as well.
\end{proof}

The following theorem answers \Cref{problem:extendable}. 
\revise{
Note that the argument can be transformed into the study of the superspline space $\mathsf{S}^{\bm{r}} : \mathcal{T} \mapsto \bm S^{\bm r}(\mathcal{T})$, and we can further deduce whether the superspline space $\mathsf{S}^{\bm{r}}$ is extendable.
}

\begin{Theorem}
    \label{thm:extendable-srk}
    The superspline space mapping \revise{$\mathsf{S}^{\bm{r}}_{k}$} is extendable \revise{only if} Assumption \eqref{eq:assumption} holds. The superspline space mapping \revise{$\mathsf{S}^{\bm{r}}$} is extendable \revise{only if} Assumption \eqref{eq:assumption2} holds.
\end{Theorem}

Before the proof, we first show some one-dimensional examples to briefly explain our idea. Here, we take triangulations 
\revise{
$\mathcal{T} = \{[-1, 0],[0,1],[1,2]\}$ and $\mathcal{T}' = \{ [-1, 0],[1,2]\} \subseteq \mathcal{T}$ on the real line $\mathbb R$.
}

\begin{Example}
    We shall prove that the mapping \revise{$\mathsf {S}^{(0)}_{0}$} is not extendable. Take $u \in \bm{S}^{(0)}_{0}(\mathcal{T}')$ such that $u|_{[1, 2]} = 1$ and $u|_{[-1, 0]} = 0$. Clearly, such $u$ cannot be extended to a function in $\bm{S}^{(0)}_{0}(\mathcal{T})$, since functions in $\bm{S}^{(0)}_{0}(\mathcal{T})$ are constant in each connected component. 

    However, for the mapping \revise{$\mathsf {S}^{(0)}_{1}$}, any function in $\bm{S}^{(0)}_{1}(\mathcal{T}')$ can be extended to $\bm{S}^{(0)}_{1}(\mathcal{T})$ by considering $v \in \mathcal{P}_{1}([0, 1])$ such that $v(1) = u(1)$ and $v(0) = u(0)$.
    
    In fact, one can show that the mapping \revise{$\mathsf {S}^{(0)}_{k}$} is extendable for $k \ge 1$, with the help of \revise{univariate} Lagrange interpolation. 
\end{Example}

\begin{Example}
    We can also prove that the mapping \revise{$\mathsf {S}^{(1)}_{2}$} is not extendable. For each $\mathcal{T}$, $\bm S_{2}^{(1)}(\mathcal{T})$ is the space of $C^{1}$ continuous piecewise quadratic functions on $\mathcal{T}$, which is known as a quadratic spline space. Take $u \in \bm{S}^{(1)}_{2}(\mathcal{T}')$ such that $u|_{[1, 2]} = 0$ and $u|_{[-1, 0]} = x$. If $u$ can be extended into a function in $\bm{S}^{(1)}_{2}(\mathcal{T})$, then in $[0, 1]$ it holds that $u(1) = u'(1) = u(0) = 0$ while $u'(0) = 1$, which leads to a contradiction. 

    Similarly, one can show that the mapping \revise{$\mathsf {S}^{(1)}_{k}$} is extendable for $k \ge 3$, with the help of \revise{univariate} Hermite interpolation.
\end{Example}

In two dimensions, we show that the mapping \revise{$\mathsf S_{k}^{(1,1)}$} is not extendable. Here, consider the following three triangles defined as
\begin{equation*}
    \begin{aligned}
        K_{++} & := \big\{(x_{1}, x_{2}) \in \mathbb{R}^2: x_{1} \ge 0, x_{2} \ge 0, x_{1} + x_{2} \le 1\big\}, \\
        K_{--} & := \big\{(x_{1}, x_{2}) \in \mathbb{R}^2: x_{1} \le 0, x_{2} \le 0, - x_{1} - x_{2} \le 1\big\}, \\
        K_{+-} & := \big\{(x_{1}, x_{2}) \in \mathbb{R}^2: x_{1} \ge 0, x_{2} \le 0, x_{1} - x_{2} \le 1\big\}.
    \end{aligned}
\end{equation*}
Let $\mathcal{T} = \{K_{++}, K_{--}, K_{+-}\}$ and $\mathcal{T}' = \{K_{++}, K_{--}\} $ be a subcomplex of $ \mathcal{T}$. Furthermore, denote the edges $F_{1} = [0, 1] \times\{0\}$, $F_{2} = \{0\} \times [- 1, 0]$, and the vertex $V_{0} = (0, 0)$, which are shown in \Cref{fig:r-2D}.

\begin{figure}[htp]
    \centering
    \includegraphics{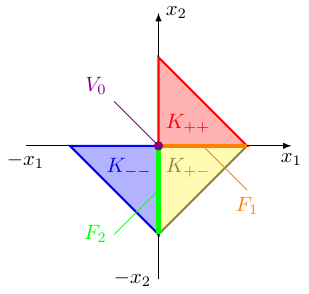}
    \caption{A mesh containing a singular vertex.}
    \label{fig:r-2D}
\end{figure}

\begin{Example}
    \label{ex:2D}
    Consider the mapping \revise{$\mathsf{S}^{(1, 1)}_{5}$} in two dimensions. \revise{For each $\mathcal{T}$, $\bm{S}^{(1, 1)}_{5}(\mathcal{T})$ is} the space of $C^{1}$ continuous piecewise $\mathcal{P}_{5}$ polynomial functions. We shall prove that the mapping \revise{$\mathsf{S}^{(1, 1)}_{5}$} is not extendable. Take \revise{a piecewise polynomial $u$} such that $u|_{K_{++}} = 0$ and $u|_{K_{--}} = x_{1}x_{2}$. Since $\partial_{x_{1}}(x_{1}x_{2})|_{V_{0}} = \partial_{x_{2}}(x_{1}x_2)|_{V_{0}} = 0$, it implies that $u \in \bm{S}^{(1, 1)}_{5}(\mathcal{T}')$. If $u$ can be extended into a function in $\bm{S}^{(1, 1)}_{5}(\mathcal{T})$, then in triangle $K_{+-}$, it holds that $\partial_{x_{2}}u|_{F_{1}} = 0$ while $\partial_{x_{1}}u|_{F_{2}} = x_{2}|_{F_{2}}$. 
    \revise{
    By comparing the coefficient of the $x_1 x_2$ term in expansion $u|_{K_{+-}} = \sum_{\sigma_{1} + \sigma_{2} \le 5} c_{\sigma} x_{1}^{\sigma_{1}} x_{2}^{\sigma_{2}}$, we immediately obtain a contradiction.
    }

    On the other hand, one can show that the mapping \revise{$\mathsf{S}^{(1, r_{2})}_{k}$} is extendable for $r_2 \ge 2$ and $k \ge 2r_{2} + 1$, with the help of Argyris element and its extension.
\end{Example}

The geometric structure around $V_{0}$ in \Cref{ex:2D} is usually considered, and the vertex $V_{0}$ is usually called a \textit{singular} vertex. 
By giving different treatments on the degrees of freedom at singular and nonsingular vertices respectively, a basis of a non-extendable superspline space can be constructed in two dimensions. 
In \cite{morgan1975nodal}, a basis is constructed for $\bm S_{k}^{(1,1)}(\mathcal{T}), k \ge 5$. 
For all vertices $V$, the degrees of freedom include the second order tangential derivatives  $\partial^2_{EE}$ along each edge $E$, sharing $V$ as an endpoint. However, for singular vertices, we need an additional degree of freedom. 

Even though a local basis of the spline function space $\bm{S}^{(1, 1)}_{5}(\mathcal{T})$ can be given from the degrees of freedom, these triangulation-dependent degrees of freedom are not considered in this paper. This implies that the existence of a local basis is not sufficient for extendability.

\section{\revise{Proof of \Cref{thm:extendable-srk}}}
\label{sec:only-if-part}

\revise{
In this section, we give a constructive proof of \Cref{thm:extendable-srk}. 
}
Some one-dimensional and two-dimensional examples have already been shown in \Cref{sec:extendable}. For general dimensions, we first prove that the condition $k \ge 2r_{d} + 1$ is necessary, by considering a subtriangulation $\mathcal{T}'$ with multiple connected components, \revise{see \Cref{prop:k-rd}.} Then we prove the necessity of the condition $r_d \ge 2r_{d - 1}$ in a specific triangulation, \revise{see \Cref{prop:continuity-rd}.} Finally, we verify the remaining part of the condition $r_{d} \ge 2r_{d - 1} \ge \cdots \ge 2^{d - 1}r_{1}$ in the same triangulation with the help of the mathematical induction argument, \revise{see \Cref{prop:continuity-rs}.}

The proof \revise{of \Cref{prop:k-rd}} is based on the following lemma, which characterizes the $C^r$ continuity in the barycentric coordinate expression. 
\revise{
We adopt the notations from \cite{de1986b} henceforth. 
} 
Given a $d$-simplex $K$ with vertices $V_{0}, V_{1}, \cdots, V_{d}$, respectively, consider the barycentric coordinate representation of a function $u \in \mathcal{P}_{k}(K)$:
\begin{equation}
    \label{eq:bb}
    u = \sum_{|\alpha| = k} c_{\alpha} \revise{\nor{\lambda}^{\alpha}},
\end{equation}
\revise{where $c_{\alpha} \in \mathbb{R}$ and $\nor{\lambda}^{\alpha} = \prod_{i = 0}^d \frac{\lambda_i^{\alpha_i}}{\alpha_i!}$ is the normalized monomial.} Here the summation goes through all multi-indices $\alpha = (\alpha_{0}, \alpha_{1}, \cdots, \alpha_{d}) \in \mathbb{N}_{0}^{d + 1}$ with the sum $|\alpha| := \sum_{i = 0}^{d} \alpha_{i} = k$, and $\lambda_{0}, \lambda_{1}, \cdots, \lambda_{d}$ are barycentric coordinates corresponding to vertices $V_{0}, V_{1}, \cdots, V_{d}$. Note that the barycentric coordinate representation is unique. 
\revise{
For a subsimplex $F$ with vertices $V_{i_0},\cdots, V_{i_s}$, denote the sums $|\alpha|_{F}$ and $|\alpha|_{\setminus F}$ by $|\alpha|_{F} = \sum_{j = 0}^{s} \alpha_{i_j}$ and $|\alpha|_{\setminus F} = |\alpha| - |\alpha|_{F}$. 
}

\revise{
The following lemma gives the vanishing condition on the coefficients, see \cite[Section 11]{de1986b}. 
\begin{Lemma}
    \label{lem:Cr-alpha}
    Fix a subsimplex $F$ of a $d$-simplex $K$, and consider the barycentric coordinate representation \eqref{eq:bb} of a function $u \in \mathcal{P}_{k}(K)$. Then, $\nabla^{n} u|_{F} = 0$ for any integer $n$ such that $0 \le n \le r$, if and only if $c_{\alpha} = 0$ for all $\alpha$ in the summation \eqref{eq:bb} such that $|\alpha|_{\setminus F} \le r$, which is equivalent to $|\alpha|_{F} \ge k - r$.
\end{Lemma}
}

We first show the necessary condition of the polynomial degree.

\begin{Proposition}
    \label{prop:k-rd}
    If the superspline space mapping \revise{$\mathsf{S}^{\bm{r}}_{k}$} is extendable, it holds that $k \ge 2r_{d} + 1$.
\end{Proposition}

\begin{proof}[Proof of \Cref{prop:k-rd}]
     Suppose that $k \le 2r_{d}$ \revise{and the superspline space mapping $\mathsf{S}^{\bm{r}}_{k}$ is extendable.} We shall construct a function $u \in \bm{S}^{\bm{r}}_{k}(\mathcal{T}')$, where $\mathcal{T}' = K_{0} \cup K_{1}$ is the union of two $d$-simplices $K_{0}$ and $K_{1}$, which cannot be extended to a function in $\bm{S}^{\bm{r}}_{k}(\mathcal{T})$ with $\mathcal{T} = K_{0} \cup K_{1} \cup K$. Here $K$ is a $d$-simplex distinct from $K_{0}$ and $K_{1}$, with the vertices $V_{0}, V_{1}, \cdots, V_{d}$. Consider $d$-simplices $K_{0}$, $K_{1}$ and $K$, such that
    \begin{equation*}
        K_{0} \cap K_{1} = \emptyset, \quad K_{0} \cap K = V_{0}, \quad K_{1} \cap K = V_{1}.
    \end{equation*}
    \revise{See \Cref{fig:k-rd} for an illustration in two dimensions.}
    \begin{figure}[htp]
        \centering
        \includegraphics{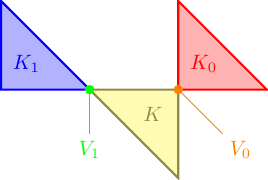}
        \caption{\revise{A two-dimensional example supporting \Cref{prop:k-rd}}.}
        \label{fig:k-rd}
    \end{figure}
    
    Since $k \le 2r_{d}$, there exist two integers $k_{0}$ and $k_{1}$ such that $k_{0} \le r_{d}$, $k_{1} \le r_{d}$, and $k = k_{0} + k_{1}$. \revise{Consider the multi-index $\beta = (k_{0}, k_{1}, 0, \cdots, 0) \in \mathbb{N}_0^{d + 1}$, where the last $(d - 1)$ indices are zero.} Define $u \in L^{2}(\mathcal{T}')$ by
    \begin{equation*}
        u|_{K_{0}} := u_{0} = 0, \quad u|_{K_{1}} := u_{1} = \revise{\nor{\lambda}^{\beta} = \frac{\lambda_{0}^{k_{0}} \lambda_{1}^{k_{1}}}{k_{0}! k_{1}!}}.
    \end{equation*}
    Here, $\lambda_{0}, \lambda_{1}, \cdots, \lambda_{d}$ are barycentric coordinates of $K$ corresponding to vertices $V_{0}, V_{1}, \cdots, V_{d}$, respectively, which are extended to linear functions on $\mathbb{R}^{d}$. Since $K_{0} \cap K_{1} = \emptyset$, it holds that the function $u \in \bm{S}^{\bm{r}}_{k}(\mathcal{T}')$.

    \revise{
    By assumption, the superspline space mapping $\mathsf{S}^{\bm{r}}_{k}$ is extendable. Therefore, 
    } 
    $u$ can be extended to a function in $\bm{S}^{\bm{r}}_{k}(\mathcal{T})$, still denoted by $u$ for convenience. Consider the \revise{restriction of $u$} on the simplex $K$, expressed uniquely as
    \begin{equation*}
        u|_{K} := u = \sum_{|\alpha| = k} c_{\alpha} \revise{\nor{\lambda}^{\alpha}}.
    \end{equation*}
    \revise{
    Now we compute the coefficients to obtain the contradiction.} Consider the vertex $V_{0}$, which is the common vertex of $K_{0}$ and $K$. By the continuity assumption, for any integer $n$ such that $0 \le n \le r_{d}$, it holds that
    \begin{equation}
        \label{eq:nablau2}
        \nabla^n u \big|_{V_{0}} = \nabla^n u_{0} \big|_{V_{0}} = 0.
    \end{equation}
    \revise{
    Observe that $|\beta|_{\setminus V_{0}} = k_{1} \le r_{d}$, then by \Cref{lem:Cr-alpha}, it holds that $c_{\beta} = 0$.
    }
    
    On the other hand, consider the vertex $V_{1}$, which is the common vertex of both $K_{1}$ and $K$. By the continuity assumption, for any integer $n$ such that $0 \le n \le r_{d}$, it holds that
    \begin{equation}
        \label{eq:nablau2-1}
        \nabla^n \big(u - \revise{\nor{\lambda}^{\beta}}\big) \big|_{V_{1}} = \nabla^n \big(u_{1} - \revise{\nor{\lambda}^{\beta}}\big) \big|_{V_{1}} = 0.
    \end{equation}
    Next, we derive the contradiction from \eqref{eq:nablau2} and \eqref{eq:nablau2-1}. Indeed, the polynomial $\big(u - \revise{\nor{\lambda}^{\beta}}\big)$ can be uniquely expressed as
    \begin{equation*}
        u - \revise{\nor{\lambda}^{\beta}} = \sum_{|\alpha| = k} c'_{\alpha} \revise{\nor{\lambda}^{\alpha}},
    \end{equation*}
    with 
    \revise{
    $c'_{\beta} = c_{\beta} - 1$. Note that $|\beta|_{\setminus V_{1}} = k_{0} \le r_{d}$, then by \Cref{lem:Cr-alpha}, it holds that $c'_{\beta} = 0$, which contradicts to $c_{\beta} = 0$. 
    } 
    As a consequence, $k \ge 2r_{d} + 1$.
\end{proof}

\revise{
\begin{Remark}
    \label{rmk:restriction}
    The contradiction in this proof arises from the function values on the line $V_{0} V_{1}$. Therefore, the results can also be derived by restricting all the function values to the one-dimensional line $V_{0} V_{1}$, and consider the one-dimensional representation of $v$, denoted as
    \begin{equation*}
        v = \sum_{\alpha_{0} + \alpha_{1} = k} c_{(\alpha_{0}, \alpha_{1})} \nor{\lambda}^{(\alpha_{0}, \alpha_{1})}.
    \end{equation*}
    The remaining proof resembles the original one, and the contradiction is derived from the coefficient $c_{(k_{0},k_{1})}$.
\end{Remark}
}

In the remaining part of this section, we fix $\mathcal{T}$ to be the triangulation $\Tq$ of the $L^{1}$ unit ball $B_{d} = \{(x_{1}, \cdots, x_{d}) \in \mathbb R^{d}: \sum_{i = 1}^{d} |x_{i}| \le 1\}$ in $\mathbb{R}^{d}$, while the triangulation $\Tq(B_{d}) = \{K_{0}, K_{1}, \cdots, K_{2^{d} - 1}\}$. The simplices $K_{j}$, $j = 0, 1, \cdots, 2^{d} - 1$ are defined as follows. For $j = 0, 1, \cdots, 2^{d} - 1$, we have the binary representation $j = \sum_{i = 1}^{d} \epsilon_{i} \cdot 2^{i - 1}, \epsilon_{i} \in \{0, 1\}$. The simplex $K_{j}$ is then defined as $K_{j} =\{(x_{1}, \cdots, x_{d}) \in B_{d}: (-1)^{\epsilon_{i}}x_{i} \ge 0 \text{ for } i = 1, \cdots, d\}$. For example, $K_{0}$ collects all the points in $B_{d}$ with non-negative coordinates, while $K_{2^{d} - 1}$ collects all the points in $B_{d}$ with non-positive coordinates.

Now let $V_{0}$ be the vertex with the coordinate $(0, \cdots, 0) \in \mathbb{R}^{d}$. For $i = 1, \cdots, d$, the vertex $V_{i}^{+}$ denotes the point with coordinate $(x_{1}, \cdots, x_{d}) \in \mathbb{R}^{d}$ such that $x_{i'} = \delta_{ii'}$, and the vertex $V_{i}^{-}$ denotes the point with coordinate $(x_{1}, \cdots, x_{d}) \in \mathbb{R}^{d}$ such that $x_{i'} = - \delta_{ii'}$, where $\delta$ is Kronecker's delta.

Therefore, the simplex $K_{j}$ has $(d + 1)$ vertices $V_{0}$, $V_{i}^{+}$ for those $\epsilon_{i} = 0$, and $V_{i}^{-}$ for those $\epsilon_{i} = 1$, for $j = \sum_{i = 1}^{d} \epsilon_{i} \cdot 2^{i - 1}, \epsilon_{i} \in \{0, 1\}$.

\revise{
Notice that the following representation is unique for a polynomial $u \in \mathcal{P}$:
\begin{equation}
    \label{eq:poly}
    u = \sum_{\sigma \in \mathbb{N}_{0}^{d}} c_{\sigma} \nor{x}^{\sigma},
\end{equation}
where $c_{\sigma} \in \mathbb{R}$, $\nor{x}^{\sigma} = \prod_{i = 1}^{d} \frac{x_{i}^{\sigma_{i}}}{\sigma_{i}!}$ is the normalized monomial, and the summation goes through all multi-indices $\sigma = (\sigma_{1}, \cdots, \sigma_{d}) \in \mathbb{N}_{0}^{d}$. For $u \in \mathcal P_k$, we also have 
\begin{equation*}
    u = \sum_{|\sigma| \le k} c_{\sigma} \nor{x}^{\sigma},
\end{equation*}
with the sum $|\sigma|$ defined as $|\sigma| := \sum_{i = 1}^{d} \sigma_{i}$.
}

We first show the following result. 
\begin{Proposition}
    \label{prop:continuity-rd}
    If the superspline space mapping \revise{$\mathsf{S}^{\bm{r}}$} is extendable, then it holds that
    \begin{equation}
        \label{eq:condition-rd}
        r_{d} \ge 2r_{d - 1}.
    \end{equation}
    In addition, the superspline space mapping \revise{$\mathsf{S}^{\bm{r}}_{k}$} is extendable only if \eqref{eq:condition-rd} holds.
\end{Proposition}

\begin{proof}
    We use a contradiction argument. Suppose that $r_{d} \le 2r_{d - 1} - 1$ and the superspline space mapping \revise{$\mathsf{S}^{\bm{r}}$} is extendable. Our goal is to prove that there exists a function $u$, defined in $\bm{S}^{\bm{r}}(K_{++} \cup K_{--})$ for two $d$-simplices $K_{++} = K_{0}$ and $K_{--} = K_{2^{d} - 1}$, but cannot be extended to $\bm{S}^{\bm{r}}(\Tq)$. Recall again that we always assume that $0 \le r_{1} \le r_{2} \le \cdots \le r_{d - 1} \le r_{d}$ and $r_{0} = 0$. Let $\overline{r}_{d - 1} = r_{d} - r_{d - 1} +1$, and $u \in L^{2}(K_{++} \cup K_{--})$ be defined by
    \begin{equation*}
        \begin{aligned}
            u|_{K_{++}} & := u_{++} = 0, \\
            u|_{K_{--}} & := u_{--} = \revise{\nor{x}^{\omega_{d}} = \frac{x_{d - 1}^{r_{d - 1}} x_{d}^{\overline{r}_{d - 1}}}{r_{d - 1}! \overline{r}_{d - 1}!}},
        \end{aligned}
    \end{equation*}
    \revise{
    where the multi-index $\omega_{d} \in \mathbb{N}_{0}^{d}$ is defined as
    }
    \begin{equation*}
        \revise{
        \omega_{d} = (\underbrace{0, \cdots, 0}_{(d - 2) \text{ zeros}}, r_{d - 1}, \overline{r}_{d - 1}).
        }
    \end{equation*}
    We first assert $u \in \bm{S}^{\bm{r}}(K_{++} \cup K_{--})$. Note that the only intersection of $K_{++}$ and $K_{--}$ is $V_{0}$, it suffices to check that
    \begin{equation*}
        \nabla^{n} u_{--} \big|_{V_{0}} = 0 = \nabla^{n} u_{++} \big|_{V_{0}}
    \end{equation*}
    for all integers $n$ such that $0 \le n \le r_{d}$, which is straightforward since $\overline{r}_{d - 1} + r_{d - 1} = r_{d} + 1 > r_{d}$.

    \begin{figure}[h]
        \centering
        \includegraphics{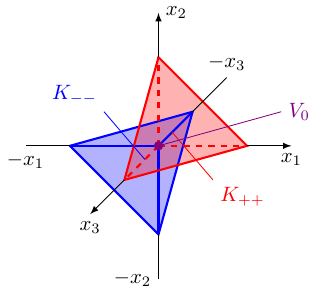}
        \qquad
        \includegraphics{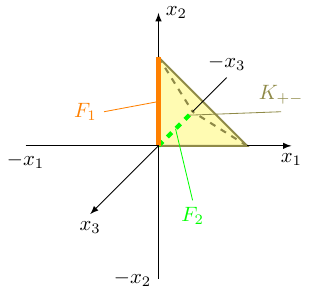}
        \caption{\revise{Three-dimensional illustrations of $\mathcal{T}'$ and $K_{+-}$.}}
    \end{figure}

    \revise{
    Since $\revise{\mathsf{S}^{\bm{r}}}$ is extendable, 
    } 
    the function $u$ can be extended to a function in $\bm{S}^{\bm{r}}(\Tq)$, still denoted by $u$ for convenience. We focus on the value of $u$ in $K_{+-} = K_{2^{d - 1}}$, namely, $u_{+-} := u|_{K_{+-}}$, and express it as
    \begin{equation*}
        u_{+-} = \revise{\sum_{\sigma \in \mathbb{N}_{0}^{d}} c_{\sigma} \nor{x}^{\sigma}}, 
    \end{equation*}

    where $c_{\sigma} \in \mathbb{R}$ and $\sigma = (\sigma_{1}, \cdots, \sigma_{d}) \in \mathbb{N}_{0}^{d}$ goes through all multi-indices. Consider the multi-index
    \begin{equation*}
        \revise{\tau_{d}(p, q)} = (\underbrace{0, \cdots, 0}_{(d - 2) \text{ zeros}}, p, q),
    \end{equation*}
    where the first $(d - 2)$ indices are zero, then it follows that
    \begin{equation*}
        \revise{\tau_{d}(r_{d - 1}, \overline{r}_{d - 1}) = \omega_{d}.}
    \end{equation*}
    
    Let $F_{1} := \{(0, \cdots, 0, x, 0): 0 \le x \le 1 \}$ be the edge on the positive $x_{d - 1}$ axis, which is a common edge of $K_{++}$ and $K_{+-}$. By the assumption, we have $\overline{r}_{d - 1} = r_{d} - r_{d - 1} + 1 \le r_{d - 1}$. It follows from the definition of $\bm{S}^{\bm{r}}(\Tq)$ that 
    \begin{equation*}
        \frac{\partial^{\overline{r}_{d - 1}}}{\partial x_{d}^{\overline{r}_{d - 1}}} u_{+-} \bigg|_{F_{1}} = \frac{\partial^{\overline{r}_{d - 1}}}{\partial x_{d}^{\overline{r}_{d - 1}}} u_{++} \bigg|_{F_{1}} = 0.
    \end{equation*}
    A direct calculation yields that
    \begin{equation*}
        \frac{\partial^{\overline{r}_{d - 1}}}{\partial x_{d}^{\overline{r}_{d - 1}}} u_{+-} \bigg|_{F_{1}} = \revise{\sum_{p \in \mathbb{N}_{0}} c_{\tau_{d}(p, \overline{r}_{d - 1})} \nor{x}^{\tau_{d}(p, 0)}} \big|_{F_{1}}.
    \end{equation*}
    \revise{
    Therefore, for $\omega_{d} = \tau_{d}(r_{d - 1}, \overline{r}_{d - 1})$, it implies that $c_{\omega_{d}} = 0$.
    }
    
    Similarly, consider the edge on the negative $x_{d}$ axis $F_{2} := \{(0, \cdots, 0, 0, x) : - 1 \le x \le 0\}$, which is a common edge of $K_{+-}$ and $K_{--}$. Again, it holds that
    \begin{equation*}
        \frac{\partial^{r_{d - 1}}}{\partial x_{d - 1}^{r_{d - 1}}} u_{+-} \bigg|_{F_{2}} = \frac{\partial^{r_{d - 1}}}{\partial x_{d - 1}^{r_{d - 1}}} u_{--} \bigg|_{F_{2}} = \revise{\nor{x}^{\tau_{d}(0, \overline{r}_{d - 1})}} \big|_{F_{2}}.
    \end{equation*}
    Another direct calculation leads to
    \begin{equation*}
        \frac{\partial^{r_{d - 1}}}{\partial x_{d - 1}^{r_{d - 1}}} u_{+-} \bigg|_{F_{2}} = \revise{\sum_{q \in \mathbb{N}_{0}} c_{\tau_{d}(r_{d - 1}, q)} \nor{x}^{\tau_{d}(0, q)}} \big|_{F_{2}}.
    \end{equation*}
    \revise{
    Hence, for $\omega_{d} = \tau_{d}(r_{d - 1}, \overline{r}_{d - 1})$, it follows that $c_{\omega_{d}} = 1$, which contradicts to $c_{\omega_{d}} = 0$. As a consequence, $r_{d} \ge 2r_{d - 1}$.
    }

    For the extendability of the superspline space mapping \revise{$\mathsf{S}^{\bm{r}}_{k}$}, a similar contradiction argument is used. Suppose $r_{d} \le 2r_{d - 1} - 1$ and the mapping \revise{$\mathsf{S}^{\bm{r}}_{k}$} is extendable. From \Cref{prop:k-rd}, it holds that $k \ge 2r_{d} + 1$. Note that $r_{d - 1} + \overline{r}_{d - 1} = r_{d} + 1 \le 2r_{d} + 1$, which implies the constructed $u \in \bm{S}^{\bm{r}}_{k}(K_{++} \cup K_{--})$. Note that $u$ cannot be extended to $\bm{S}^{\bm{r}}(\Tq)$ and $\bm{S}^{\bm{r}}_{k}(\Tq) \subseteq \bm{S}^{\bm{r}}(\Tq)$, which leads to a contradiction. Therefore, it can be asserted that the superspline space mapping \revise{$\mathsf{S}^{\bm{r}}_{k}$} is extendable only if $r_{d} \ge 2r_{d - 1}$ holds.
\end{proof}

We now prove the general result.

\begin{Proposition}
    \label{prop:continuity-rs}
    If the superspline space mapping \revise{$\mathsf{S}^{\bm{r}}$} is extendable, then it holds that
    \begin{equation}
        \label{eq:condition-rs}
        r_{d} \ge 2r_{d - 1} \ge \cdots \ge 2^{d - 1}r_{1}. \tag{A2}
    \end{equation}
    In addition, the superspline space mapping \revise{$\mathsf{S}_{k}^{\bm{r}}$} is extendable only if \eqref{eq:condition-rs} holds.
\end{Proposition}

\begin{proof}
    This proof is based on a mathematical induction. The first inequality of \eqref{eq:condition-rs} for $r_{d}$ and $r_{d - 1}$ is already proved in \Cref{prop:continuity-rd}. Suppose for $t = d, d - 1, \dots, s + 1$ with $s \ge 2$, it holds that $r_{t} \ge 2r_{t - 1}$, but $r_{s} < 2r_{s - 1}$ and the superspline space mapping \revise{$\mathsf{S}^{\bm{r}}$} is extendable.
    
    Let the two $d$-simplices $K_{++} = K_{0}$, $K_{--} = K_{2^s - 1}$, and let $\overline{r}_{s - 1} := r_{s} - r_{s - 1} + 1$. Define $u \in L^{2}(K_{++} \cup K_{--})$ by 
    \begin{equation*}
        \begin{aligned}
            u|_{K_{++}} := u_{++} & = 0, \\
            u|_{K_{--}} := u_{--} & \revise{= \nor{x}^{\omega_{s}}} \bigg(1 - \sum_{i = s + 1}^{d} x_{i}\bigg)^{r_{d} - r_{s}} \\
            & = \revise{\bigg(\frac{x_{s - 1}^{r_{s - 1}} x_{s}^{\overline{r}_{s - 1}}}{r_{s - 1}! \overline{r}_{s - 1}!} \prod_{i = s + 1}^{d} \frac{x_{i}^{r_{i} - r_{i - 1}}}{(r_{i} - r_{i - 1})!}\bigg)} \bigg(1 - \sum_{i = s + 1}^{d} x_{i}\bigg)^{r_{d} - r_{s}},
        \end{aligned}
    \end{equation*}
    \revise{
    where the multi-index $\omega_{s} \in \mathbb{N}_{0}^{d}$ is defined as
    }
    \begin{equation*}
        \revise{\omega_{s} = (\underbrace{0, \cdots, 0}_{(s - 2) \text{ zeros}}, r_{s - 1}, \overline{r}_{s - 1}, r_{s + 1} - r_{s}, \cdots, r_{d} - r_{d - 1}).}
    \end{equation*}
    
    We first show that $u \in \bm{S}^{\bm{r}}(K_{++} \cup K_{--})$. Note that all the common subsimplices of the two simplices $K_{++}$ and $K_{--}$ lie in 
    \begin{equation*}
        F_{12} := \big\{ (\underbrace{0,\cdots, 0}_{\text{$s$ zeros }} , x_{s + 1}, \cdots, x_{d}) \in B_{d} : x_{s + 1}, \cdots, x_{d} \ge 0\big\}.
    \end{equation*}
    This $F_{12}$ is the convex hull of $V_{s + 1}^{+}, \cdots, V_{d}^{+}$ and $V_{0}$. It follows from $r_{s} \le 2r_{s - 1} - 1$ that $\overline{r}_{s - 1} \le r_{s - 1}$. 
     
    For a $(d - t)$-dimensional subsimplex $E$ of $F_{12}$, we separate the discussion by whether $E$ contains $V_0$. 
    \begin{itemize}
        \item[-] The vertices of $E$ are $V_{i_{t+1}}^{+}, \cdots, V_{i_{d}}^{+}$ and $V_{0}$. Note here $x_{s - 1}$, $x_{s}$ and $x_{i'}$ for $V_{i'}^{+} \not\in E$ vanish on $E$. Due to $r_{d} - r_{d - 1} \ge r_{d - 1} - r_{d - 2} \ge \cdots \ge r_{s + 1} - r_{s}$, it can be deduced that $u_{--}$ has a factor
        \begin{equation*}
            x_{s - 1}^{r_{s - 1}} x_{s}^{\overline{r}_{s - 1}} \prod_{i' > s, V_{i'}^{+} \not\in E} x_{i'}^{r_{i'} - r_{i' - 1}}
        \end{equation*}
        with a total degree not less than $r_{s - 1} + \overline{r}_{s - 1} + \sum_{l = s + 1}^t (r_l - r_{l - 1}) = r_{t} + 1$. Therefore, for any integer $n$ such that $0 \le n \le r_{t}$, it holds that
        \begin{equation*}
            \nabla^n u_{--} \big|_{E} = 0 = \nabla^n u_{++} \big|_{E}.
        \end{equation*}
        
        \item[-] The vertices of $E$ are $V_{i_{t}}^{+}, \cdots, V_{i_{d}}^{+}$. Note here $x_{s - 1}$, $x_{s}$ and $\big(1 - \sum_{i = s + 1}^{d} x_{i}\big)$ vanish on $E$, and it follows from $r_{s - 1} + \overline{r}_{s - 1} + (r_{d} - r_{s}) = r_{d} + 1 > r_{t}$ that for any integer $n$ such that $0 \le n \le r_{t}$, it holds that
        \begin{equation*}
            \nabla^n u_{--} \big|_{E} = 0 = \nabla^n u_{++} \big|_{E}.
        \end{equation*}
    \end{itemize}
    As a consequence, we have $u \in \bm{S}^{\bm{r}}(K_{++} \cup K_{--})$.

    \begin{figure}[htp]
        \centering
        \includegraphics{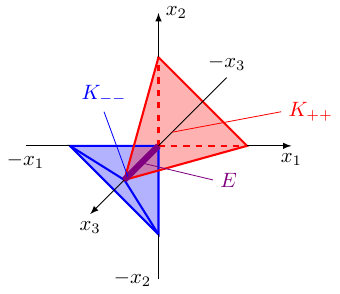}
        \qquad
        \includegraphics{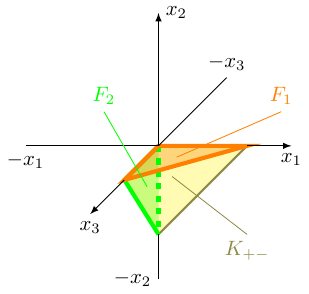}
        \caption{\revise{Three-dimensional illustrations of $\mathcal{T}'$ and $K_{+-}$.}}
    \end{figure}
    
    \revise{
    Since $\revise{\mathsf{S}^{\bm{r}}}$ is extendable, 
    } 
    the function $u$ can be extended to a function in $\bm{S}^{\bm{r}}(\Tq)$, still denoted by $u$ for convenience. We focus on the value of $u$ in $K_{+-} = K_{2^{s - 1}}$, namely, $u_{+-} := u|_{K_{+-}}$, and express it as
    \begin{equation}
        u_{+-} = \revise{\sum_{\sigma \in \mathbb{N}_{0}^{d}} c_{\sigma} \nor{x}^{\sigma}},
    \end{equation}
    where $c_{\sigma} \in \mathbb{R}$ are coefficients determined later, and $\sigma = (\sigma_{1}, \cdots, \sigma_{d}) \in \mathbb{N}_{0}^{d}$ goes through all multi-indices. Consider the multi-index
    \begin{equation*}
        \tau_{s}(p, q, \rho) = (\underbrace{0, \cdots, 0}_{(s - 2) \text{ zeros}}, p, q, \rho_{s + 1}, \cdots, \rho_{d}).
    \end{equation*}
    \revise{
    Here $\rho := (\rho_{s + 1}, \cdots, \rho_{d}) \in \mathbb{N}_{0}^{d - s}$. Let the multi-index $\overline{\rho}_{s} := (r_{s + 1} - r_{s}, \cdots, r_{d} - r_{d - 1}) \in \mathbb{N}_{0}^{d - s}$, then it holds that
    }
    \begin{equation*}
        \revise{\tau_{s}(r_{s - 1}, \overline{r}_{s - 1}, \overline{\rho}_{s}) = \omega_{s}.}
    \end{equation*}
   
    Consider the common subsimplex of $K_{++}$ and $K_{+-}$, denoted as 
    \begin{equation}
        F_{1} := \{(\underbrace{0, \cdots, 0}_{(s - 2) \text{ zeros}}, x_{s - 1}, 0, x_{s + 1},\cdots, x_{d}) \in B_{d}: x_{s - 1}, x_{s + 1}, \cdots, x_{d} \ge 0\}. 
    \end{equation}
    Note that $F_{1}$ is the convex hull of $V_{s - 1}^{+}$, $V_{s + 1}^{+}, \cdots, V_{d}^{+}$ and $V_{0}$. By the continuity assumption, it implies that
    \begin{equation*}
        \frac{\partial^{\overline{r}_{s - 1}}}{\partial x_{s}^{\overline{r}_{s - 1}}} u_{+-} \bigg|_{F_{1}} = \frac{\partial^{\overline{r}_{s - 1}}}{\partial x_{s}^{\overline{r}_{s - 1}}} u_{++} \bigg|_{F_{1}} = 0.
    \end{equation*}
    It follows a direct calculation that
    \begin{equation*}
        \frac{\partial^{\overline{r}_{s - 1}}}{\partial x_{s}^{\overline{r}_{s - 1}}} u_{+-} \bigg|_{F_{1}} = \revise{\sum_{p \in \mathbb{N}_{0}, \rho \in \mathbb{N}_{0}^{d - s}} c_{\tau_{s}(p, \overline{r}_{s - 1}, \rho)} \nor{x}^{\tau_{s}(p, 0, \rho)}} \big|_{F_{1}}.
    \end{equation*}
    Therefore, 
    \revise{
    for all $p \in \mathbb{N}_{0}, \rho \in \mathbb{N}_{0}^{d - s}$, it holds that the coefficients $c_{\tau_{s}(p, \overline{r}_{s - 1}, \rho)} = 0$. Since $\omega_{s} = \tau_{s}(r_{s - 1}, \overline{r}_{s - 1}, \overline{\rho}_{s})$, it follows that $c_{\omega_{s}} = 0$.
    }

    On the other hand, consider the common subsimplex of $K_{--}$ and $K_{+-}$, namely,
    \begin{equation}
        F_{2} := \{(\underbrace{0, \cdots, 0}_{(s - 2) \text{ zeros}}, 0, - x_{s}, x_{s + 1},\cdots, x_{d}) \in B_{d}: x_{s}, x_{s + 1}, \cdots, x_{d} \ge 0\}. 
    \end{equation}
    Note that $F_{2}$ is the convex hull of $V_{s}^{-}$, $V_{s + 1}^{+}, \cdots, V_{d}^{+}$ and $V_{0}$. By the continuity assumption, it implies that
    \begin{align*}
        \frac{\partial^{r_{s - 1}}}{\partial x_{s - 1}^{r_{s - 1}}} u_{+-} \bigg|_{F_{2}} & = \frac{\partial^{r_{s - 1}}}{\partial x_{s - 1}^{r_{s - 1}}} u_{--} \bigg|_{F_{2}} \\
        & = \revise{\nor{x}^{\tau_{s}(0, \overline{r}_{s - 1}, \overline{\rho}_{s})}} \bigg(1 - \sum_{i = s + 1}^{d} x_{i}\bigg)^{r_{d} - r_{s}}\bigg|_{F_{2}} \\
        & = \revise{\nor{x}^{\tau_{s}(0, \overline{r}_{s - 1}, \overline{\rho}_{s})}} \big|_{F_{2}} + \text{ high order terms.}
    \end{align*}
    By a direct calculation, it holds that
    \begin{equation*}
        \frac{\partial^{r_{s - 1}}}{\partial x_{s - 1}^{r_{s - 1}}} u_{+-} \bigg|_{F_{2}} = \revise{\sum_{q \in \mathbb{N}_{0}, \rho \in \mathbb{N}_{0}^{d - s}} c_{\tau_{s}(r_{s - 1}, q, \rho)} \nor{x}^{\tau_{s}(0, q, \rho)}} \big|_{F_{2}}.
    \end{equation*}
    \revise{
    Now it suffices to compare the coefficient of the term $\nor{x}^{\tau_{s}(0, \overline{r}_{s - 1}, \overline{\rho}_{s})}$. Since it holds that $\omega_{s} = \tau_{s}(r_{s - 1}, \overline{r}_{s - 1}, \overline{\rho}_{s})$, it leads to $c_{\omega_{s}} = 1$, which contradicts to $c_{\omega_{s}} = 0$. 
    } 
    Therefore, the superspline space mapping \revise{$\mathsf{S}^{\bm{r}}$} is not extendable if $r_{s} < 2r_{s - 1}$.

    Finally, we conclude the result by an induction argument on $s$.
    
    Again, for the extendability of the superspline space mapping \revise{$\mathsf{S}^{\bm{r}}_{k}$}, a similar contradiction argument is used. Suppose for each $t = d, d - 1, \dots, s + 1$, $s \ge 2$, it holds that $r_{t} \ge 2r_{t - 1}$, but $r_{s} < 2r_{s - 1}$ and the superspline space mapping \revise{$\mathsf{S}^{\bm{r}}_{k}$} is extendable. From \Cref{prop:k-rd}, it holds that $k \ge 2r_{d} + 1$. Note that $r_{s - 1} + \overline{r}_{s - 1} + \sum_{i = s + 1}^{d} (r_{i} - r_{i - 1}) + (r_{d} - r_{s}) = 2r_{d} - r_{s} + 1 \le 2r_{d} + 1$, which implies the constructed $u \in \bm{S}^{\bm{r}}_{k}(K_{++} \cup K_{--})$. Note that $u$ cannot be extended to $\bm{S}^{\bm{r}}(\Tq)$ and $\bm{S}^{\bm{r}}_{k}(\Tq) \subseteq \bm{S}^{\bm{r}}(\Tq)$, which leads to a contradiction. Therefore, it can be asserted that the superspline space mapping \revise{$\mathsf{S}^{\bm{r}}_{k}$} is extendable only if \eqref{eq:condition-rs} holds.
\end{proof}

\revise{
\begin{Remark}
    \label{rmk:restriction-rs}
    Similar to \Cref{rmk:restriction}, the contradictions in the proofs of \Cref{prop:continuity-rd} and \Cref{prop:continuity-rs} comes from the function values on a $(d - s + 2)$-dimensional plane spanned by  $x_{s - 1}, x_{s}, \cdots, x_{d}$.
\end{Remark}
}

\revise{
\subsection{Macro superspline spaces in two dimensions}

\label{subsec:Macro}

In this subsection, we show that the superspline mapping and its extendability can be generalized to the analysis of macro elements, indicating some necessary conditions on smoothness when constructing $C^r$ macro elements. 
Let $\mathcal{T}$ be a triangulation. For each simplex $K$ of $\mathcal{T}$ we apply a specific refinement procedure $a$ to generate a macro element $K_{a}$. The resulting refined triangulation is denoted as $\mathcal{T}_a$.
The superspline space on $\mathcal{T}$ (not $\mathcal{T}_{a}$!) can then be defined as: $\bm{S}^{\bm{r}}_a(\mathcal{T}) := \bm{S}^{\bm r}(\mathcal{T}_a)$ and $ \bm{S}^{\bm{r}}_{a,k}(\mathcal{T}) := \bm{S}^{\bm r}_k(\mathcal{T}_a)$.

Now we turn to the discussion of the extendability of the macro superspline mapping $\mathsf S_{a,k}^{\bm r} :\mathcal{T} \mapsto \bm S_{a,k}^{\bm r}(\mathcal{T})$. This provides several essential insights into the necessary conditions for constructing finite elements based on macro superspline spaces (referred to as macro element spaces in \cite{lai2007spline}).
However, deriving these necessary conditions for each kind of macro elements in general proves to be highly non-trivial.

Here, we focus on several planar macro elements. 
Using computer algebra system Macaulay2 \cite{M2}, we obtain the following results:

\begin{enumerate}
    \item[-] For Alfeld split,  $\mathsf S_{A,k}^{(2,3)}$ is extendable for $k \ge 7$, while $\mathsf S_{A,k}^{(2,2)}$ is not extendable for $k\geq 4$, and  $\mathsf S_{A,k}^{(3,3)}$ is not extendable for $k\geq 5$.
    \item[-] For Morgan-Scott split, $\mathsf S_{MS,k}^{(2,2)}$ is extendable for $k \ge 6$,  $\mathsf S_{MS,k}^{(2,3)}$ is extendable for $k \ge 7$, while $\mathsf S_{MS,k}^{(3,3)}$ is not extendable for $k\geq 5$.
\end{enumerate}
For the Alfeld split, the result matches the smoothness condition illustrated in \cite{lai2007spline}. Additionally, some necessary conditions based on the B-form representation are also discussed \cite{laghchim1994triangular}. 
The general cases for the smoothness condition on macro elements are beyond the scope of this paper.

One of the implications is that the extendability is strongly related to the local configuration. For instance, all the contradictions we derived in this article rely on $\Tq$, the triangulation of the $L^{1}$ unit ball. However, some of the obstructions of the extendability will disappear if some configuration (singularity) does not exist. The macro superspline space mapping represents one of the most extreme cases. In this case, we resolve the singularity through manual refinement. In the literature, there are lots of results about how the local geometry affects the dimensions of the spline spaces and the superspline spaces. Therefore, it seems natural and enlightening to discuss how local geometry affects the algebraic property (extendability in this paper) of spline space mappings and superspline space mappings. This is left as a central part of the future work.
}

\appendix
\section{Details of the finite element spaces and the coincidence with the superspline spaces}
\label{sec:if-part}

Throughout this section, we suppose that the continuity vector $\bm{r} = (r_{1}, \cdots, r_{d})$ and the polynomial degree $k$ satisfy Assumption \eqref{eq:assumption}, and let $r_{0} = 0$ for convenience. This section proves that under Assumption \eqref{eq:assumption}, the global finite element space $\bm{E}^{\bm{r}}_{k}(\mathcal{T})$ constructed in \cite{hu2021construction} is coincident with the superspline space $\bm{S}^{\bm{r}}_{k}(\mathcal{T})$. Therefore, the corresponding finite element space mapping $\mathsf{E}^{\bm{r}}_{k}$ is the superspline space mapping $\mathsf{S}^{\bm{r}}_{k}$.

We first recall the degrees of freedom of the finite element space from \cite{hu2021construction}. 

\begin{Definition}[Bubble function spaces]
    Let $0 \le s \le d$, $F \in \mathcal{T}_{d - s}$ be a $(d - s)$-dimensional subsimplex, and $ 0 \le n \le r_s$. The bubble function space $\mathcal{B}_{F, n, k}$ on $F$, with $\bm{r} = (r_{1}, \cdots, r_{d})$ and  $k$, is defined as
    \begin{equation*}
        \mathcal{B}_{F, n, k} := \Span \bigg\{\prod_{i = s}^d \lambda_{F, i}^{\sigma_{i}}: (\sigma_{s}, \cdots, \sigma_{d}) \text{ satisfies \eqref{eq:condi-sigma-1} and \eqref{eq:condi-sigma-2}} \bigg\}.
    \end{equation*}
    Here $\lambda_{F, s}, \cdots, \lambda_{F, d}$ are barycentric coordinates of $F$, and the multi-index $(\sigma_{s}, \cdots, \sigma_{d}) \in \mathbb{N}_{0}^{d - s + 1}$ satisfies
    \begin{equation}
        \label{eq:condi-sigma-1}
        \sum_{i = s}^{d} \sigma_{i} = k - n,
    \end{equation}
    and
    \begin{equation}
        \label{eq:condi-sigma-2}
        \begin{gathered}
            \sigma_{i_{1}} + \cdots + \sigma_{i_{l}} > r_{l + s} - n, \\
            \forall \{i_{1}, \cdots, i_{l}\} \subsetneq \{s, \cdots, d\}, \quad 1 \le l \le d - s.
        \end{gathered}
    \end{equation}
    For the special case $s = d$, i.e., $F$ is a vertex $V$, we set $\mathcal{B}_{V, n, k} := \mathbb{R}$.
\end{Definition}

The dependence on $\bm r$ and $k$ will be omitted if the context is clear. Based on the bubble function space $\mathcal{B}_{F, n, k}$, we can define the degrees of freedom on $F$.
 
\begin{Definition}[Degrees of Freedom]
    \label{defi:dof}
    Let $0 \le s \le d$, and $F \in \mathcal{T}_{d - s}$ be a $(d - s)$-dimensional subsimplex. The degrees of freedom on $F$ are defined as follows. For each $0 \le n \le r_{s}$ and $\theta \in \{(\theta_{1}, \cdots, \theta_{s}) \in \mathbb{N}_{0}^{s}: |\theta| := \sum_{i = 1}^{s} \theta_{i} = n\}$, define the functional space for $u$ by weighted moments as
    \begin{equation*}
        \label{eq:dof-FEM}
        \frac{1}{\abs{F}} \int_{F}(D_{F}^{\theta} u) \cdot v, \quad \forall v \in \mathcal{B}_{F, n, k}.
    \end{equation*}
    Here $D_{F}^{\theta} := \frac{\partial^{n}}{\prod_{i = 1}^{s}\partial \bm{n}_{F, i}^{\theta_{i}}} $ represents an $n$-th order normal derivative of $u$ on $F$, where $\bm{n}_{F, 1}, \cdots, \bm{n}_{F, s}$ are $s$ pairwise orthogonal unit normal vector(s) of $F$.
\end{Definition}

The following technical lemma provides some connection between bubble function spaces.
\begin{Lemma}
    \label{lem:bubble-coin}
    Let $\mathcal{B}_{E, n'' + n, k}$ is defined with respect to  $\bm{r} = (r_{1}, \cdots, r_{d})$ and $k$, and $\widetilde{\mathcal{B}}_{E, n'', k - n}$ is defined with respect to $\bm{q} = (r_{s + 1} - n, \cdots, r_{d} - n)$ and $(k - n)$. It holds that  $
        \mathcal{B}_{E, n'' + n, k} = \widetilde{\mathcal{B}}_{E, n'', k - n}.$
\end{Lemma}

\begin{proof}
    Let $\lambda_{E, t}, \cdots, \lambda_{E, d}$ be barycentric coordinates of $E$. 
Note that $\widetilde{\mathcal{B}}_{E, n'', k - n}$ 
    \begin{equation*}
        \widetilde{\mathcal{B}}_{E, n'', k - n} = \Span \left\{\prod_{i = t}^d \lambda_{E, i}^{\sigma_{i}}: (\sigma_{t}, \cdots, \sigma_{d}) \text{ satisfies \eqref{eq:condi-sigma-1-2} and \eqref{eq:condi-sigma-2-2}} \right\}.
    \end{equation*}
    Here, $(\sigma_{t}, \cdots, \sigma_{d}) \in \mathbb{N}_{0}^{d - t + 1}$ satisfies
    \begin{equation}
        \label{eq:condi-sigma-1-2}
        \sum_{i = t}^d \sigma_{i} = (k - n) - n'' = k - (n'' + n),
    \end{equation}
    and
    \begin{equation}
        \label{eq:condi-sigma-2-2}
        \begin{gathered}
            \sigma_{i_{1}} + \cdots + \sigma_{i_{m}} > q_{l + t - s} - n'' = r_{m + t} - (n'' + n), \\
            \forall \{i_{1}, \cdots, i_{m}\} \subsetneq \{t, \cdots, d\}, \quad 1 \le m \le d - t.
        \end{gathered}
    \end{equation}
    Similarly, 
    \begin{equation*}
        \mathcal{B}_{E, n'' + n, k} = \Span \left\{\prod_{i = t}^d \lambda_{E, i}^{\sigma_{i}}: (\sigma_{t}, \cdots, \sigma_{d}) \text{ satisfies \eqref{eq:condi-sigma-1-1} and \eqref{eq:condi-sigma-2-1}} \right\},
    \end{equation*}
    where, $(\sigma_{t}, \cdots, \sigma_{d}) \in \mathbb{N}_{0}^{d - t + 1}$ satisfies
    \begin{equation}
        \label{eq:condi-sigma-1-1}
        \sum_{i = t}^d \sigma_{i} = k - (n'' + n),
    \end{equation}
    and
    \begin{equation}
        \label{eq:condi-sigma-2-1}
        \begin{gathered}
            \sigma_{i_{1}} + \cdots + \sigma_{i_{l}} > r_{l + t} - (n'' + n), \\
            \forall \{i_{1}, \cdots, i_{l}\} \subsetneq \{t, \cdots, d\}, \quad 1 \le l \le d - t.
        \end{gathered}
    \end{equation}
    Comparing \eqref{eq:condi-sigma-1-2} and \eqref{eq:condi-sigma-1-1}, \eqref{eq:condi-sigma-2-2} and \eqref{eq:condi-sigma-2-1}, one can get the conclusion.
\end{proof}

By the definition of the finite element spaces, a finite element space over a triangulation $\mathcal{T}$ is defined as
\begin{equation}
    \label{eq:FE-space}
    \begin{aligned}
        \bm{E}^{\bm{r}}_{k}(\mathcal{T}) & := \{u\in L^{2}(\Omega) : \text{for each } d \text{-simplex } K \in \mathcal{T}, u|_{K} \in \mathcal{P}_{k}(K); \\
        & \qquad \text{for each } 1 \le s \le d, F \in \mathcal{T}_{d - s}, 0 \le n \le r_{s}, |\theta| = n,  v \in \mathcal{B}_{F, n, k} \\
        & \qquad \frac{1}{|F|} \int_{F} \big(D_{F}^{\theta} u|_{K}\big) \cdot v \text{ is single-valued for any } K \in \Star(F; \mathcal{T})\}.
    \end{aligned}
\end{equation}
Clearly, $\bm S^{\bm r}_k(\mathcal T)$ is a subspace of $\bm E^{\bm r}_{k}(\mathcal T)$. 

We now prove that the two spaces are the same. The result is based on the following proposition.

\begin{Proposition}
    \label{prop:Adj-continuity}
    For each $s$ such that $1 \le s \le d$, let $F$ be a $(d - s)$-dimensional subsimplex, shared by two adjacent $d$-dimensional simplices $K_{+}$ and $K_{-}$. Suppose $u_{+} \in \mathcal{P}_{k}(K_{+})$ and $u_{-} \in \mathcal{P}_{k}(K_{-})$, satisfying that $l(u_{+}) = l(u_{-})$ holds for each degree of freedom $l$ defined on $F$ and its subsimplices. Then for any integer $n$ such that $0 \le n \le r_{s}$, it holds that
    \begin{equation*}
        \nabla^{n} u_{+} \big|_{F} = \nabla^{n} u_{-} \big|_{F}.
    \end{equation*}
\end{Proposition}

\begin{proof}
    Clearly, it suffices to check the cases for the normal derivatives. Since for any tangential derivative  $D_{/\!/}$ of $F$, $u_{+}|_{F} = u_{-}|_{F}$ implies $D_{/\!/} u_{+}|_{F} = D_{/\!/} u_{-}|_{F}$. Therefore, it is sufficient to prove that for each $\theta \in \mathbb{N}_{0}^{s}$ with $|\theta| = n \le r_{s}$, it holds that
    \begin{equation*}
        D_{F}^{\theta} u_{+} \big|_{F} = D_{F}^{\theta} u_{-} \big|_{F}.
    \end{equation*}

    For any $t$ such that $s \le t \le d$, let $E$ be a $(d - t)$-dimensional subsimplex of $F$. For any integer $n'$ such that $0 \le n' \le r_{t}$ and any multi-index $\theta' \in \mathbb{N}_{0}^{t}, |\theta'| = n'$, it holds that
    \begin{equation}
    \label{eq:De}
        \frac{1}{\lvert E \rvert}\int_{E} \big(D_{E}^{\theta'} u_{+}\big) \cdot v = \frac{1}{\lvert E \rvert}\int_{E} \big(D_{E}^{\theta'} u_{-}\big) \cdot v, \quad \forall v \in \mathcal{B}_{E, n', k}.
    \end{equation}
    Here $D_{E}^{\theta'} u := \frac{\partial^{n'}}{\prod_{i = 1}^{t} \partial \bm{n}_{E, i}^{\theta'_{i}}} u$, where $\bm{n}_{E, 1}, \cdots, \bm{n}_{E, t}$ are $t$ pairwise orthogonal unit normal vectors of $E$. 

    Since $E$ is a subsimplex of $F$ and $\bm{n}_{F, 1}, \cdots, \bm{n}_{F, s}$ are orthogonal unit normal vectors of $F$, these vectors are also the normal vectors of $E$. Hence, $\bm{n}_{F, 1}, \cdots, \bm{n}_{F, s}$ can be linearly represented by $\bm{n}_{E, 1}, \cdots, \bm{n}_{E, t}$. 
    Therefore, given $n$ and $\theta \in \mathbb{N}_{0}^{t}, |\theta| =  n'' \le r_{t} - n$, it follows from \eqref{eq:De} that 
    \begin{equation}
        \label{eq:dof-deri}
        \frac{1}{\lvert E \rvert}\int_{E} \big[D_{E}^{\theta''} \big(D_{F}^{\theta} u_{+} \big)\big] \cdot v = \frac{1}{\lvert E \rvert}\int_{E} \big[D_{E}^{\theta''} \big(D_{F}^{\theta} u_{-} \big)\big] \cdot v, \quad \forall v \in \mathcal{B}_{E, n'' + n, k}.
    \end{equation}

    We shall prove that \eqref{eq:dof-deri} implies
    \begin{equation*}
        D_{F}^{\theta} u_{+}\big|_{F} = D_{F}^{\theta} u_{-}\big|_{F},
    \end{equation*}
    which follows from the fact that $D_{F}^{\theta} u_{+}\big|_{F}$ and $D_{F}^{\theta} u_{-}\big|_{F}$ conform to the degrees of freedom \Cref{defi:dof} of $\bm E_{k-n}^{\bm q}(F)$ with $\bm{q} = (r_{s + 1} - n, \cdots, r_{d} - n)$.

    Following \Cref{lem:bubble-coin}, we use $\mathcal{B}$ and $\widetilde{\mathcal{B}}$ to distinguish the two sets of bubble function spaces, where $\mathcal{B}$ is for those spaces with respect to $\bm{r}$ and $k$, and $\widetilde{\mathcal{B}}$ for these spaces with respect to $\bm{q}$ and $(k - n)$. The degrees of freedom of $\bm{E}_{k - n}^{\bm q}(F)$ are
    \begin{equation}
        \label{eq:dof-subsimplex}
       w \mapsto \frac{1}{\lvert E \rvert}\int_{E} \big(D_{E}^{\theta''} w\big) \cdot v, \quad \forall v \in \widetilde{\mathcal{B}}_{E, n'', k - n}. 
    \end{equation}
    It follows from \Cref{lem:bubble-coin} that $\mathcal{B}_{E, n'' + n, k} = \widetilde{\mathcal{B}}_{E, n'', k - n}$.
 
    As a consequence, a combination of the equation \eqref{eq:dof-deri} and the degrees of freedom \eqref{eq:dof-subsimplex} leads to $D_{F}^{\theta} u_{+}\big|_{F} = D_{F}^{\theta} u_{-}\big|_{F}$.
\end{proof}

\begin{Proposition}
    \label{coro:E=S}
    Suppose $\bm{r}$ and $k$ satisfy \eqref{eq:assumption}. For any triangulation $\mathcal{T}$, it holds that $\bm{E}^{\bm{r}}_{k}(\mathcal{T}) = \bm{S}^{\bm{r}}_{k}(\mathcal{T})$. That is,  $\mathsf{E}^{\bm{r}}_{k} = \mathsf{S}^{\bm{r}}_{k}$.
\end{Proposition}
\begin{proof}
    By \Cref{prop:Adj-continuity}, $\bm{E}^{\bm{r}}_{k}(\mathcal{T})$ is a subspace of $\bm{S}^{\bm{r}}_{k}(\mathcal{T})$. The other side of inclusion is indicated by the definition of $\bm E_{k}^{\bm r}(\mathcal T)$. Therefore, we conclude the result. 
\end{proof}
\bibliographystyle{plain}
\bibliography{references}

\end{document}